\documentclass{amsart}
\usepackage{amsmath, amssymb, latexsym}

\def\Xint#1{\mathchoice
  {\XXint\displaystyle\textstyle{#1}}%
  {\XXint\textstyle\scriptstyle{#1}}%
  {\XXint\scriptstyle\scriptscriptstyle{#1}}%
  {\XXint\scriptscriptstyle\scriptscriptstyle{#1}}%
  \!\int}
\def\XXint#1#2#3{{\setbox0=\hbox{$#1{#2#3}{\int}$}
    \vcenter{\hbox{$#2#3$}}\kern-.5\wd0}}

\def\dashint{\Xint-}

\newtheorem{Thm}{Theorem}[section]
\newtheorem{Prop}[Thm]{Proposition}
\newtheorem{Cor}[Thm]{Corollary}
\newtheorem{Lem}[Thm]{Lemma}
\newtheorem{Conj}[Thm]{Conjecture}
\newtheorem{Rem}[Thm]{Remark}
\numberwithin{equation}{section}

\begin{document}
\title
{The $p$-Royden and $p$-harmonic boundaries for Metric Measure Spaces}

\author[M. Lucia]{Marcello Lucia}
\address{Department of Mathematics \\College of Staten Island-CUNY \\ 2800 Victory Boulevard \\ Staten Island \\NY 10314 
\\USA}
\email{Marcello.Lucia@csi.cuny.edu}

\author[M.J. Puls]{Michael J. Puls}
\address{Department of Mathematics \\John Jay College-CUNY \\ 524 West 59th Street \\ New York \\NY 10019 
\\USA}
\email{mpuls@jjay.cuny.edu}
\thanks{The research of the first author was partially supported by Simons Foundation grant 210368.\\
The research of the second author was partially supported by PSC-CUNY grant 66269-00 44}

\begin{abstract}
Let $p$ be a real number greater than one and let $X$ be a locally compact, noncompact metric measure space that satisfies certain conditions. The $p$-Royden and $p$-harmonic boundaries of $X$ are constructed by using the $p$-Royden algebra of functions on $X$ and a Dirichlet type problem is solved for the $p$-Royden boundary. We also characterize the metric measure spaces whose $p$-harmonic boundary is empty.
\end{abstract}

\keywords{Dirichlet problem at infinity, metric measure space, $p$-harmonic function, $p$-parabolic, $p$-Royden algebra, $p$-weak upper gradient, $(p,p)$-Sobolev inequality}
\subjclass[2010]{Primary: 31B20; Secondary: 31C25, 54E45}

\date{May 19, 2015}
\maketitle

\section{Introduction}\label{introduction}
Throughout this paper $p$ will always denote a real number greater than one. Let $\Omega$ be a domain in the complex plane $\mathbb{C}$. The {\em extended boundary} of $\Omega$ is the usual boundary if $\Omega$ is bounded and is the usual boundary along with the point at infinity if it is unbounded. The domain $\Omega$ is known as a Dirichlet domain if the {\em Dirichlet problem} is solvable on $\Omega$. That is, if $f$ is a continuous real-valued function on the extended boundary of $\Omega$, then there exists a function $h$ that is harmonic in the interior of $\Omega$ and is equal to $f$ on the extended boundary of $\Omega$. It is well known that the unit disk in $\mathbb{C}$ is a Dirichlet domain. In fact, any simply connected domain in $\mathbb{C}$ is a Dirichlet domain.

More recently, Dirichlet type problems have been investigated in the more general setting of a metric measure space $X$. A good introduction to this topic is \cite[Chapter 10]{2Bjornbook}. With some additional assumptions on $X$, the following theorem is proved in \cite[Theorem 10.24]{2Bjornbook}
\begin{Thm}\label{Dirichletboundedmms}
Let $\Omega$ be a bounded domain in $X$ and assume that the Sobolev capacity of $X\setminus \Omega$ is greater than zero. If $f$ is a continuous function on $\partial \Omega$, then there exists a unique bounded $p$-harmonic function $h$ in $\Omega$ such that
\[ \lim_{\Omega \ni y \rightarrow x} h(y) = f(x) \text{ for quasieverywhere } x \in \partial \Omega. \]
\end{Thm}
A natural question to ask is what can we say about the Dirichlet problem if $\Omega$ is unbounded? Shanmugalingam showed in \cite[Proposition 5.3]{Shanmug03} that with some assumptions on the metric space $X$, and if the measure on $X\setminus \Omega$ is positive, then for every function $f \in N^{1,p}(X)$ (Newtonian space) there is a solution to the $p$-Dirichlet problem on $\Omega$ with boundary data $f$. In \cite{Hansevi} Hansevi proved a similar result for the more general obstacle problem.  However, the papers \cite{Hansevi, Shanmug03} do not take into account the behavior of $f$ at infinity, or apply in the case $\Omega = X$. The papers \cite{Hansevi, Shanmug03} also illustrate that the main issue we come up against is that there is no natural boundary for $X$ that corresponds to the extended boundary of $\mathbb{C}$. What we need is a compactification for $X$ that will allow us to define a suitable boundary for $X$. The authors of \cite{HoloLangVaha07} use the Gromov boundary to study the Dirichlet problem at infinity for Gromov hyperbolic metric measure spaces. Another possible compactification for $X$ is the $p$-Royden compactification. This compactification allows us to define the $p$-Royden boundary of $X$. The $p$-harmonic boundary of $X$ is a special subset of the $p$-Royden boundary. The purpose of this paper is to construct both the $p$-harmonic boundary and the $p$-Royden boundary for a metric measure space. We prove a Dirichlet type problem at infinity for $X$ using the $p$-Royden boundary, and we characterize the $p$-parabolicity of $X$ in terms of the $p$-harmonic boundary.

In \cite{Royden62} Royden introduced the harmonic boundary of a Riemann surface $R$, also see Chapter III of \cite{SarioNakai70}. It was also shown in \cite{Royden62} that a bounded harmonic function with finite Dirichlet integral on a open Riemann surface $R$ can be determined by its behavior on the harmonic boundary of $R$. The results for Riemann surfaces were extended to noncompact orientable Riemannian manifolds in \cite{GlasnerKatz70}. In \cite{Lee00} the concept of the harmonic boundary was generalized to the $p$-harmonic boundary. The harmonic boundary corresponds to the case $p=2$. In \cite{Lee00} Lee proved the following Dirichlet type result: Suppose $M$ is a complete Riemannian manifold of bounded geometry and that the $p$-harmonic boundary of $M$-consists of $n$ points, where $n$ is finite. Then every bounded $p$-harmonic function on $M$ with finite $p$-Dirichlet integral is determined uniquely by its value on the $p$-harmonic boundary. Consequently, there is a nonconstant bounded $p$-harmonic function with finite $p$-Dirichlet integral on $M$ if and only if the $p$-harmonic boundary of $M$ consists of more than one point. 
In \cite{Puls10} many of the results in \cite{Lee00, Lee05} were shown to be true in the setting of graphs with bounded degree.

In 1975 Yau \cite{Yau75} proved that on a complete Riemannian manifold of non-negative Ricci curvature, every positive harmonic function is constant. Since then a number of papers have appeared studying various Liouville type problems, not only in the linear setting of harmonic functions, but also in the nonlinear setting of $p$-harmonic functions. See the introduction of \cite{HoloLangVaha07} for an excellent review of work done related to these issues. One reason for studying Dirichlet type problems at infinity is to determine when a subclass of $p$-harmonic functions on a space are constant or not. This of course, tells us if the space has a Liouville type property or not.

In this paper $X$ will always be a locally compact, noncompact, complete metric measure space with metric $d$ and positive complete Borel measure $\mu$. Furthermore, we will also assume that $\mu$ is doubling and that if $B$ is a nonempty open ball in $X$, then $0 < \mu(B) < \infty$. It will be assumed throughout that $X$ contains at least two points. Note that $X$ is also proper since we are assuming that it is complete and $\mu$ is doubling. Recall that a metric space is proper if all its closed and bounded subsets are compact. The main result of this study is the following theorem. All unexplained notation will be defined in later parts of the paper. 

\begin{Thm} \label{mainresult} 
Let $X$ be a metric measure space. Suppose $X$ satisfies a $(1,p)$-Poincar\'e inequality and a $(p, p)$-Sobolev inequality. Furthermore, assume that the volume of all balls of a fixed positive radius is bounded below by a positive constant. Let $f$ be a continuous real-valued function on the $p$-Royden boundary of $X$. Then there exists a real-valued $p$-harmonic function $h$ on $X$, such that $\lim_{n \rightarrow \infty} h(x_n) = f(x)$ whenever $x$ is an element of the $p$-Royden boundary and $(x_n)$ is a sequence in $X$ converging to $x$.
\end{Thm}

In Section \ref{preliminaries} we define many of the terms that will be used throughout the paper. We also define an algebra of functions on $X$, the $p$-Royden algebra, which is critical for defining the $p$-Royden boundary. We end Section \ref{preliminaries} by giving the definition for a function to be $p$-harmonic on $X$. Section \ref{defnpharmboundary} is devoted to the construction of the $p$-Royden and $p$-harmonic boundaries of $X$. We also give a characterization of when the $p$-harmonic boundary is empty. In Section \ref{proofmainresult} we prove Theorem \ref{mainresult}. We speculate in Section \ref{someremarks} about what would happen if we dropped the $(p,p)$-Sobolev inequality condition from Theorem \ref{mainresult}.

We would like to thank the referee for making some useful suggestions that led to significant improvements to the paper.

\section{Preliminaries}\label{preliminaries}
In this section we will define some terms and notation that will be used throughout the paper. We also define the $p$-Royden algebra for $X$. This algebra is crucial for defining the $p$-Royden boundary of $X$. We end this section by giving the definition of a $p$-harmonic function on $X$.

Denote respectively by $\mathbb{R}$ and $\mathbb{N}$ the real numbers and the natural numbers. If $x \in X$ and $r$ is a positive real number, then $B_r(x)$ will denote the metric ball of radius $r$ centered at $x$. We will write a.e. to indicate that a given property holds almost everywhere with respect to the measure $\mu$. A connected open set is a domain. If $\Omega$ is a domain in $X$, then $\overline{\Omega}$ will denote the closure of $\Omega$ in $X$. A measurable real-valued function $f$ on $X$ is $p$-integrable if $\int_X \vert f \vert^p\, d\mu < \infty$. Let $\mathcal{L}^p(X)$ be the set of extended real-valued measurable functions on $X$ that are $p$-integrable. Observe that in $\mathcal{L}^p(X)$ we do not identify those functions that differ only on a set of measure zero. We will write $L^p(X)$ to indicate the set obtained from $\mathcal{L}^p(X)$ by identifying functions that agree a.e.. The usual Banach space norm on $L^p(X)$ will be denoted by $\Vert \cdot \Vert_p$. The notation $\Omega \Subset X$ will mean that $\Omega$ is a relatively compact subset of $X$. The space $\mathcal{L}^p_{loc} (X)$ will consist of those functions $f$ that satisfy $f \in \mathcal{L}^p(\Omega)$ for all $\Omega \Subset X$. For a extended real-valued function $f$ on $X$, let $\Vert f \Vert_{\infty} = \sup\{ \vert f(x) \vert \mid x \in X\}$ and let $\mbox{supp}\,f$ denote the support of $f$ on $X$.

The main ingredient used in the construction of the $p$-Royden boundary of $X$ is the $p$-Royden algebra of functions on $X$.  For Riemannian manifolds and graphs the $p$-Royden algebra consists of all bounded continuous functions $f$ such that $\vert \nabla f \vert$, where $\nabla f$ denotes the gradient of $f$, is $p$-integrable. This poses a problem for us because there is no differentiable structure on a general metric measure space. It turns out that the norm of $\nabla f$ is what is essential. Over the past couple of decades the theory of the $p$-weak upper gradient of a function $f$ on a metric measure space has been developed as a generalization of $\vert \nabla f \vert$. In particular, $p$-weak upper gradients have been very useful in the development of Newtonian spaces, which are Sobolev type spaces on metric measure spaces. See \cite[Chapter 1]{2Bjornbook} and the references therein for information concerning Newtonian spaces.

Let $\mathcal{P}$ be the family of all locally rectifiable curves in $X$. For $\Gamma \subset \mathcal{P}$ let $Q(\Gamma)$ be the set of all Borel functions $\rho \colon X \rightarrow [0, \infty]$ that satisfy 
\[ \int_{\gamma} \rho ds \geq 1 \mbox{ for every } \gamma \in \Gamma. \]
The {\em $p$-modulus} of $\Gamma$ is defined to be 
\[ \mbox{Mod}_p (\Gamma) = \inf_{\rho \in Q(\Gamma)} \int_X \rho^p d\mu. \]
We shall say that a property of curves holds for {\em $p$-almost every curve} if the family of curves for which the property fails has zero $p$-modulus.

Suppose $u \colon X \rightarrow \mathbb{R}$ is a Borel function. A Borel function $g \colon X \rightarrow [0, \infty]$ is defined to be an {\em upper gradient} of $u$ if 
\[ \vert u(\gamma(b)) - u(\gamma(a)) \vert \leq \int_{\gamma} g ds \]
for every rectifiable curve $\gamma \colon [a, b] \rightarrow X$. Note that $ g = \infty$ is an upper gradient for $u$. We shall say that $g$ is a {\em $p$-weak upper gradient} of $u$ if the above inequality holds on $p$-almost every curve $\gamma$ in $\mathcal{P}$. It is worth mentioning that if $g$ is a $p$-weak upper gradient for $u$ and $v$ is a function on $X$ that satisfies $u = v$ a.e., then it is not necessarily true that $g$ is a $p$-weak upper gradient of $v$. However, what is true is that if $g'$ is a nonnegative function on $X$ for which $g' = g$ a.e., then $g'$ is also a $p$-weak upper gradient for $u$, \cite[Corollary 1.44]{2Bjornbook}. The following proposition is a direct consequence of the fact $\text{Mod}_p(\bigcup_{j=1}^{\infty} \Gamma_j) \leq \sum_{j=1}^{\infty} \text{Mod}_p(\Gamma_j)$, which is \cite[Lemma 1.34(b)]{2Bjornbook}.

\begin{Prop}\label{linear}
Let $u, v$ be extended real-valued functions on $X$ and let $a$ and $b$ be real numbers. Suppose that $g$ and $h$ are $p$-weak upper gradients for $u$ and $v$ respectively. Then $\vert a \vert g + \vert b \vert h$ is a $p$-weak upper gradient for $au + bv$.
\end{Prop}

The following Proposition will be needed in the sequel. It's slightly more general then \cite[Proposition 2.3]{2Bjornbook}, but the proof is exactly the same. 
\begin{Prop}\label{completeness}
Assume that $(u_n)$ is a sequence of real-valued Borel functions on $X$ and that $g_n \in L^p(X)$ is a $p$-weak upper gradient of $u_n$ for each $n \in \mathbb{N}$. Also assume that $u_n \rightarrow u$ a.e., $g_n \rightarrow g$ in $L^p(X)$ and that $g$ is nonnegative. Then there is a function $\tilde{u} = u$ a.e. such that $g$ is a $p$-weak upper gradient of $\tilde{u}$.

Moreover, if there is a subsequence $(u_{n,k})$ of $(u_n)$ such that $(u_{n,k}) \rightarrow u$ q.e., then we may choose $\tilde{u} = u.$
\end{Prop}

We shall say that $g$ is a {\em minimal $p$-weak upper gradient} for $u$ if $g_1 \geq g$ a.e. for any $p$-weak upper gradient $g_1$ of $u$. The following theorem is \cite[Theorem 2.5]{2Bjornbook}. 

\begin{Thm}\label{minimumpweak}
Let $1 < p \in \mathbb{R}$. If $u \in \mathcal{L}^p_{loc}(X)$ and $u$ has a $p$-weak upper gradient in $L^p(X)$, then there exists a minimal $p$-weak upper gradient $g$ of $u$. Moreover, $g$ is unique up to sets of measure zero.
\end{Thm}

We mentioned earlier that one of the ingredients needed in constructing the $p$-Royden boundary on a manifold or a graph is the the norm of the gradient of a function be $p$-integrable. One of the problems we encounter is that there is no differentiable structure on a metric measure space, so we cannot compute the gradient of a function. It turns out though that the minimal $p$-weak upper gradient is an adequate replacement for the norm of the gradient of a function. With this in mind we will write $\vert \nabla u \vert$ to indicate the minimal $p$-weak upper gradient of $u$.

There is no guarantee that if $\vert \nabla u \vert =0$, then $u$ is constant. In order to rectify this situation we need to assume that $X$ satisfies a $(1,p)$-Poincar\'e inequality, which we now describe. Let $q \geq 1$. We shall say that $X$ satisfies a $(q, p)$-Poincar\'e inequality if there exists constants $C > 0$ and $\sigma \geq 1$ such that 
\[ \left( \dashint_{B_r(x)} \vert u - u_{B_r(x)} \vert^q d\mu \right)^{1/q} \leq rC \left(\dashint_{B_{\sigma r}(x)} g^p d\mu\right)^{1/p} \]
for all metric balls $B_r(x), u \in \mathcal{L}^p_{loc}(X)$ and $g$ a $p$-weak upper gradient of $u$, where
\[ u_{B_r(x)} = \dashint_{B_r (x)} u d\mu = \frac{1}{\mu(B_r(x))} \int_{B_r(x)} u d\mu. \]

For the rest of this paper we will assume that $X$ satisfies the $(1,p)$-Poincar\'e inequality. By \cite[Proposition 4.2]{2Bjornbook} $X$ is connected. Thus the $(1,p)$-Poincar\'e inequality assumption on $X$ implies that $g=0$ is a $p$-weak upper gradient of $u$ if and only if $u$ is constant a.e..

Define $BD^p(X)$ to be the set of bounded continuous functions on $X$ with minimal $p$-weak upper gradient in $L^p(X)$. An immediate consequence of Proposition \ref{linear} is that $BD^p(X)$ is a vector space with respect to pointwise addition of functions and scalar multiplication.  Furthermore, $BD^p(X)$ is closed under pointwise multiplication. To see this let $u, v \in BD^p(X).$ Then 
\[ \vert u(\gamma(b)) - u(\gamma(a)) \vert \leq \int_{\gamma} \vert \nabla u \vert ds \]
for all nonconstant rectifiable curves $\gamma \colon [a,b] \rightarrow X$ not in $\Gamma_u$, where $\mbox{Mod}_p(\Gamma_u) = 0$. Also, there exists a set $\Gamma_v$ of rectifiable curves such that $\mbox{Mod}_p (\Gamma_v) = 0$ and 
\[ \vert v(\gamma(b)) - v(\gamma(a)) \vert \leq \int_{\gamma} \vert \nabla v \vert ds \]
for all nonconstant rectifiable curves not in $\Gamma_v$. Now \cite[Lemma 1.34(b)]{2Bjornbook} says that $\mbox{Mod}_p (\Gamma_u \cup \Gamma_v) = 0$. Let $\gamma \colon [a,b] \rightarrow X$ be a nonconstant rectifiable curve that does not belong to $\Gamma_u \cup \Gamma_v$. Then
\begin{eqnarray*} 
\vert (uv)(\gamma(b)) - (uv)(\gamma(a))\vert & \leq & \| u \|_{\infty} \vert v(\gamma(b)) - v(\gamma(a)) \vert + \| v \|_{\infty} \vert u(\gamma(b)) - u(\gamma (a)) \vert \\
   & \leq & \int_{\gamma} ( \Vert u \Vert_{\infty} \vert \nabla v \vert + \Vert v \Vert_{\infty} \vert \nabla u \vert ) ds.
\end{eqnarray*}
Thus, $\Vert u \Vert_{\infty} \vert \nabla v \vert + \Vert v \Vert_{\infty} \vert \nabla u \vert$ is a $p$-weak upper gradient for $uv$ and it is also in $L^p(X)$. Hence, $uv \in BD^p(X)$. The algebra $BD^p(X)$ is known as the {\em $p$-Royden algebra of $X$}.

Let $(u_n)$ be a sequence of functions and $u$ a function on $X$. We shall say that $(u_n) \rightarrow u$ in the $CD^p$-topology if
\[ \limsup_K \vert u_n - u \vert \rightarrow 0 \text{ for all compact subsets }K \text{ of } X \]
and
\[ \int_X \vert \nabla(u_n - u) \vert^p d\mu \rightarrow 0. \]
If the sequence $(u_n)$ is uniformly bounded, in addition to the above conditions, then we will say that $(u_n) \rightarrow u$ in the $BD^p$-topology.
\begin{Thm}\label{closed}
Let $1 < p \in \mathbb{R}$ and let $(u_n)$ be a sequence in $BD^p(X)$ that is uniformly bounded on $X$. Suppose $u$ is a real-valued function on $X$ and that $(u_n) \rightarrow u$ in the $BD^p$-topology. Then $u \in BD^p(X)$.
\end{Thm}
\begin{proof}
Let $(u_n)$ be a sequence in $BD^p(X)$ that is uniformly bounded on $X$ and suppose $(u_n) \rightarrow u$ in the $BD^p$-topology. Since $X$ is locally compact and $(u_n) \rightarrow u$ uniformly on compact sets, $u$ is a continuous function on $X$. Also $u$ is bounded on $X$ due to $(u_n)$ being uniformly bounded on $X$. Using the techniques from the proof of \cite[Theorem 1.5]{GoldTroyanov01}, or from the proofs of \cite[Lemma 3.6, Theorem 3.7]{Shanmug00}, a Cauchy sequence $(g_n)$ in $L^p(X)$ can be constructed where $g_n$ is a $p$-weak upper gradient for $u_n$ for each $n$. Denote the limit of $(g_n)$ in $L^p(X)$ by $g$, where $g \geq 0$ on $X$. Since $(u_n) \rightarrow u$ pointwise, Proposition \ref{completeness} says that $g$ is a $p$-weak upper gradient of $u$ in $L^p(X)$. Therefore, $u \in BD^p(X)$ and the proof of the theorem is complete.
\end{proof}

Before we move on to the definition of a $p$-harmonic function we need to define the Sobolev $p$-capacity of a set in $X$. Let $N^{1,p}(X)$ be the set of functions $u \in \mathcal{L}^p(X)$ that have a $p$-weak upper gradient in $L^p(X)$. A seminorm can be defined on $N^{1,p}(X)$ via 
 \[ \Vert u \Vert_{N^{1,p}(X)} = \left( \int_X \vert u \vert^p d\mu + \int_X \vert \nabla u \vert^p d\mu \right)^{1/p}. \]
The space $N^{1,p}(X)$ is known as a Newtonian space, and was originally studied by Shanmugalingam in \cite{Shanmug00}. Newtonian spaces were developed in order to establish a Sobolev space type theory on metric measure spaces.

Now suppose $E \subset X$. The {\em Sobolev $p$-capacity} of $E$ is the number
\[ C_p(E) = \inf \Vert u \Vert_{N^{1,p}(X)}, \]
where the infimum is taken over all $u \in N^{1,p}(X)$ for which $u \geq 1$ on $E$. We shall say that a property $P$ holds {\em quasieverywhere} (q.e.) if the set of points on which $P$ fails has Sobolev $p$-capacity zero. One nice fact about Sobolev $p$-capacity is that two functions which agree q.e. on $X$ have the same set of $p$-weak upper gradients, \cite[Corollary 1.49]{2Bjornbook}. Compare this to the fact we mentioned earlier that two functions that agree a.e. do not necessarily have the same set of $p$-weak upper gradients.

Define $N^{1,p}_{loc} (X)$ to be the set consisting of all real-valued functions $f$ on $X$ with $f \in \mathcal{L}^p_{loc} (X)$ and $\vert \nabla f \vert \in \mathcal{L}^p_{loc}(X)$. Identify functions on $N^{1,p}_{loc} (X)$ that agree q.e.. For an open set $U$ in $X$ let $C_0(U)$ denote the set of functions that equal zero q.e. on $X\setminus U$. We shall say that a function $h \in N^{1,p}_{loc} (X)$ is a {\em $p$-minimizer} in $X$ if 
\[ \int_{\Omega} \vert \nabla h \vert^p\, d\mu \leq \int_{\Omega} \vert \nabla v \vert^p\, d\mu  \]
holds for every open $\Omega \Subset X$ and every $v \in N^{1,p}_{loc} (X)$ for which $h - v \in C_0(\Omega)$. The function $h$ is said to be {\em $p$-harmonic} on $X$ if it is a continuous $p$-minimizer on $X$. We will write $HBD^p(X)$ to indicate the $p$-harmonic functions contained in $BD^p(X)$.

\section{The $p$-Royden and $p$-harmonic boundaries}\label{defnpharmboundary}

In this section we define the $p$-Royden and $p$-harmonic boundaries of $X$ by constructing an appropriate compactification of $X$. The algebra $BD^p(X)$ is crucial for this construction. We finish the section by characterizing the metric measure spaces whose $p$-harmonic boundary is the empty set. We begin with
\begin{Lem}\label{seperate}
The space $BD^p(X)$ separates points from closed sets in $X$.
\end{Lem}
\begin{proof}
Let $A$ be a closed set in $X$ and let $x \in X\setminus A$. Pick $\epsilon > 0$ such that $B_{2\epsilon}(x) \cap A = \emptyset$. Define $u \colon X \rightarrow \mathbb{R}$ by
\[ u(y) = \left\{ \begin{array}{rl}
                               d(x,y),  &  y \in B_{\epsilon}(x) \\
                              \epsilon, & y \notin B_{\epsilon} (x) \end{array} \right. . \]
Since $u(x) = 0 \notin \overline{u(A)}$, $u$ seperates $x$ from $A$. The proof of the lemma will be complete once we show $u \in BD^p(X)$, which we now do. Define $g \colon X \rightarrow [0, \infty]$ by
\[ g(y) = \left\{ \begin{array}{cl}
       1    &    y \in B_{2\epsilon} (x)  \\
      0    &     y  \notin B_{2\epsilon}(x)  \end{array} \right. . \]
Clearly $g \in L^p(X)$. Let $\gamma \colon [a,b] \rightarrow X $ be a rectifiable curve and suppose that $\gamma(a)$ and $\gamma(b)$ are elements of $B_{\epsilon}(x)$. It follows from the triangle inequality that $d( \gamma(b), x) - d(\gamma(a),x) \leq d(\gamma(a), \gamma(b))$. Hence
\[ \vert d(\gamma(b), x) - d(\gamma(a), x) \vert = \vert u(\gamma(b)) - u(\gamma(a)) \vert \leq \int_{\gamma} g ds. \]
Similar calculations show that this inequality is also true in the cases $\gamma(b) \in B_{\epsilon}(x), \gamma(a) \notin B_{\epsilon}(x)$ and both $\gamma(b), \gamma(a) \notin B_{\epsilon}(x)$. Therefore, $g$ is an upper gradient, and hence a $p$-weak upper gradient of $u$. Thus, $u \in BD^p(X)$ and the proof of the lemma is now complete. 
\end{proof}

For $u \in BD^p(X)$ let $I_u$ be a closed bounded interval in $\mathbb{R}$ that contains the image of $u$. Let $Y$ denote the product space
\[ Y \colon= \prod_{u \in BD^p(X)} I_u \] 
 with the Tychonoff topology. The space $Y$ can be thought of as the set of real-valued functions with domain $BD^p(X)$. Furthermore, $Y$ is a compact Hausdorff space, and a sequence $(x_n)$ converges to $x$ in $Y$ if $(x_n(f))$ converges to $x(f)$ for all $f \in BD^p(X).$ The evaluation map $e \colon X \rightarrow Y$ is given by
\[ e(x) f = f(x). \]
We saw in Lemma \ref{seperate} that $BD^p(X)$ separates points from closed sets in $X$, so \cite[Theorem 8.16]{Willard70} tells us that $e$ is actually an embedding of $X$ into $Y$. We identify $X$ with $e(X)$. Let $\overline{X} = \overline{e(X)}$, where the closure is taken in $Y$. Thus, $X$ is an open dense subset of the compact set $\overline{X}$. Also, every function in $BD^p(X)$ can be extended to a continuous function on $\overline{X}$. Denote by $C(\overline{X})$ the set of continuous functions on $\overline{X}$ with the uniform norm. By the Stone-Wierstrass theorem, $BD^p(X)$ is dense in $C(\overline{X})$.

Set $R_p(X) = \overline{X} \setminus X$. The compact Hausdorff space $R_p(X)$ is known as the {\em $p$-Royden boundary} of $X$. We will write $BD^p_c(X)$ to indicate the set of functions in $BD^p(X)$ that have compact support. Denote by $\overline{BD^p_c(X)}_{BD^p}$ the closure of $BD^p_c(X)$ with respect to the $BD^p$-topology. It follows from Theorem \ref{closed} that $\overline{BD^p_c(X)}_{BD^p}$ is contained in $BD^p(X)$. The {\em $p$-harmonic boundary} of $X$ is the following subset of $R_p(X)$:
\[ \Delta_p (X) \colon= \{ x \in R_p(X) \mid x(u) = 0 \text{ for all } u \in \overline{BD^p_c(X)}_{BD^p} \}. \]

Sometimes it will be the case $\Delta_p(X) = \emptyset$. The following theorem will be useful in determining when this happens. 

\begin{Thm}\label{equalsoneempty}
Let $F$ be a closed subset of $\overline{X}$ such that $F \cap \Delta_p(X) = \emptyset$. Then there exists a $u \in \overline{BD^p_c (X)}_{BD^p}$ such that $u =1$ on $F$ and $0 \leq u \leq 1$.
\end{Thm}
\begin{proof}
Let $x \in F$. Since $x \notin \Delta_p(X)$ there exists $u_x \in \overline{BD^p_c (X)}_{BD^p}$ such that $u_x (x) \neq 0$. Replacing $u_x$ by $-u_x$ if need be, we assume that $u_x(x) >0$. Pick a neighborhood $U_x$ of $x$ for which $u_x > 0$ on $U_x$. Using $\max\{0, u_x\}$ instead of $u_x$ if necessary, we also assume that $u_x \geq 0$ on $X$. Since $F$ is compact, there exists $x_1, \dots, x_k$ that satisfy $F \subseteq \cup_{i =1}^k U_{x_i}$. Let
\[ g = \sum_{i =1}^k u_{x_i}. \]
Then $ g \in \overline{BD^p_c (X)}_{BD^p}$. Set $c = \inf\{ g(x) \mid x \in F\}$. 
So $c > 0$. Now let
\[ u = \min\{1, c^{-1}g \}.\]
Since $\vert \nabla(c^{-1}g) \vert \in L^p(X)$ is a $p$-weak upper gradient of $u$, we have that $u \in BD^p(X)$. In fact, $u \in \overline{BD^p_c(X)}_{BD^p}$. Indeed, let $(g_n)$ be a sequence in $BD^p_c(X)$ that converges to $g$ in the $BD^p$-topology. Define $u_n \in BD^p_c(X)$ by 
\[ u_n = \min\{1, c^{-1}g_n\}. \]
Then $\vert \nabla(u - u_n) \vert \leq \vert \nabla(c^{-1}g - c^{-1}g_n) \vert$ since $\vert \nabla(c^{-1}g - c^{-1}g_n ) \vert$ is a $p$-weak upper gradient of $u- u_n$. Thus
\[  \int_X \vert \nabla( u - u_n)\vert^p d\mu \leq \int_X \vert \nabla (g - g_n) \vert^p d\mu. \]
Hence, $\int_X \vert \nabla(u - u_n) \vert^p d\mu \rightarrow 0$ as $n \rightarrow \infty$. Because $(g_n) \rightarrow g$ uniformly on compact sets, it follows that $(u_n) \rightarrow u$ uniformly on compact sets. Therefore, $u \in \overline{BD^p_c(X)}_{BD^p}$ and $0 \leq u \leq 1$ on $X$.
\end{proof}

Let $1_X$ denote the function that equals one for all $x \in X$. Since $\vert \nabla(1_X) \vert = 0, 1_X \in BD^p(X)$. We shall say that $X$ is {\em $p$-parabolic} if $1_X \in \overline{BD^p_c(X)}_{BD^p}$. If $X$ is not $p$-parabolic, then it is said to be {\em $p$-hyperbolic}. A consequence of the above theorem is the following characterization for $\Delta_p(X)$. 
\begin{Cor} \label{parabolicchar}
Let $X$ be a metric measure space. Then $X$ is $p$-parabolic if and only if $\Delta_p(X) = \emptyset.$
\end{Cor}
\begin{proof}
If $\Delta_p(X) = \emptyset$, then by the above theorem $1_X \in \overline{BD^p_c(X)}_{BD^p}$ and $X$ is $p$-parabolic.

Conversely, suppose $X$ is $p$-parabolic and assume that $\Delta_p(X) \neq \emptyset$. Let $x \in \Delta_p(X)$. Then there exists a sequence $(x_n)$ in $X$ such that $ (x_n(f)) \rightarrow x(f)$ for each $f \in BD^p(X)$. Since $x_n(1_X) = 1_X(x_n) = 1$ for all $n, x(1_X) = 1$. However, this contradicts our hypothesis that $1_X \in \overline{BD^p_c(X)}_{BD^p}$. Hence, $\Delta_p(X) = \emptyset$.
\end{proof}
\section{Proof of Theorem \ref{mainresult}} \label{proofmainresult}
In this section we prove Theorem \ref{mainresult}. We start by giving a crucial lemma that is a slightly modified version of \cite[Theorem 10.24]{2Bjornbook}. We then use the lemma in the proof of a proposition that is Theorem \ref{mainresult} for the special case $f \in BD^p(X)$. With this result in hand, we use an approximation argument to prove our main result.

Recall that $X$ represents a metric measure space that is locally compact, noncompact, complete, satisfies the $(1,p)$-Poincar\'e inequality and whose measure is doubling.

\begin{Lem} \label{Dirichletrelcompact}
Let $\Omega$ be a relatively compact domain of $X$. Suppose $f \in BD^p(X)$. Then there exists an unique $p$-harmonic function $h$ in $\Omega$, such that  $h = f$ q.e. on $\partial\Omega$ and $\vert \nabla h \vert \in L^p(X)$.
\end{Lem}
\begin{proof} 
Let $f \in BD^p(X)$. By \cite[Theorem 10.24]{2Bjornbook} there exists a unique bounded $p$-harmonic function $h$ in $\Omega$ such that 
\[ \lim_{\Omega \ni y \rightarrow x} h(y) = f(x) \text{ for  } x \in \partial \Omega\setminus E, \]
where $E \subseteq \partial \Omega$ and has Sobolev $p$-capacity zero, so $h = f$ q.e. on $\partial \Omega$.
Extend $h$ to all of $X$ by setting $h = f$ on $X\setminus \overline{\Omega}$. It now follows from $\int_{\Omega} \vert \nabla h \vert^p d\mu \leq \int_{\Omega} \vert \nabla f \vert^p d\mu$ and $\vert \nabla f \vert \in L^p(X)$ that $\vert \nabla h \vert \in L^p(X)$. 
\end{proof}
Since $X$ is connected (the $(1,p)$-Poincar\'e inequality implies connected), second countable and locally compact, there exists an exhaustion $(\Omega_k)$ of $X$ by relatively compact domains. 
\begin{Prop}\label{harmoniconx}
Let $f \in BD^p(X)$ and let $(\Omega_k)$ be an exhaustion of $X$. Let $h_k$ be the unique $p$-harmonic function on $\Omega_k$ that satisfies $f = h_k$ q.e. on $X\setminus \Omega_k$. Then there exists a subsequence $(h_{i,k})$ of $(h_k)$ that converges locally uniformly to a $p$-harmonic function $h \in HBD^p(X)$.
\end{Prop}
\begin{proof}
Let $(\Omega_k)$ be an exhaustion of $X$ and let $f \in BD^p(X)$. Let $(h_k)$ be the sequence on $X$ for which $h_k$ is $p$-harmonic on $\Omega_k$ and also satisfies $h_k =f$ q.e. on $X\setminus \Omega_k$. Combining the strong maximum principle, \cite[Theorem 8.13]{2Bjornbook}, with  $f \in BD^p(X)$ we obtain a constant $M$ that is a uniform bound for the sequence $(h_k)$. Let $j \in \mathbb{N}$. It follows from \cite[Theorem 8.15]{2Bjornbook} that
\[ \vert h_n (x) - h_n(y) \vert \leq C2Md(x,y)^{\alpha}, \]
where $x,y \in \Omega_j$ and $n > j$. Furthermore, the constant $C$ does not depend on $n$, and $0 < \alpha < 1.$ Thus the family $(h_n)_{n > j}$ is equicontinuous on $\Omega_j$. For $j=1$, the Ascoli-Arzela theorem yields a subsequence $(h_{1,k})$ of $(h_k)$ such that $(h_{1,k})$ converges uniformly on $\Omega_1$ to a continuous function $v_1$. Now there exists a subsequence $(h_{2,k})$ of $(h_{1,k})$ such that $(h_{2,k})$ converges uniformly on $\Omega_2$ to a continuous function $v_2$. Continue inductively in this manner for each $j$. The diagonal construction produces a subsequence $(h_{i,k})$ of $(h_k)$ such that $(h_{i,k})$ converges to a continuous function $h$ on $X$. The sequence $(h_{i,k})$ converges locally uniformly on $X$ because $(h_{j,k})$ converges uniformly on each $\Omega_j$ and $(h_{i,k})$ is a subsequence of $(h_{j,k})$ for each $j$ and $i \geq j$. By \cite[Theorem 9.36]{2Bjornbook}, $h$ is $p$-harmonic on $X$. The boundedness of $h$ follows from the fact that $(h_k)$ is uniformly bounded on $X$. 

We now complete the proof of the proposition by showing $\vert \nabla h \vert \in L^p(X)$. Relabel the subsequence $(h_{i,k})$ as $(h_k)$. The sequence $( \vert \nabla h_k \vert )$ is bounded in $L^p(X)$ because $\vert \nabla h_k \vert \leq \vert \nabla f \vert$ for all $k$. So by passing to a subsequence if necessary, $( \vert \nabla h_k \vert )$ converges weakly to $g \in L^p(X)$. By Mazur's lemma there exists a sequence of finite linear convex combinations $\vert \overline{ \nabla h_i }\vert = \sum_{k = i}^{N_i} a_{i,k} \vert \nabla h_i \vert$ such that $\Vert \vert \overline{\nabla{h_i}}\vert - g \Vert_p \rightarrow 0$. Set $\overline{h_i} = \sum_{k = i}^{N_i} a_{i,  k} h_i$. By Proposition \ref{linear} $\vert \overline{\nabla h_i} \vert$ is a $p$-weak upper gradient for $\overline{h_i}$. As $i \rightarrow \infty, \overline{h_i}$ converges pointwise to $h$, so $g$ is a $p$-weak upper gradient for $h$ by Proposition \ref{completeness}. Hence $\vert \nabla h \vert \in L^p(X)$.
\end{proof}

Since $h \in BD^p(X)$ it can be extended to $\overline{X}$. A reasonable question to ask is for what $x \in R_p(X) (\Delta_p(X))$ does $f(x) = h(x)$?

We now define  conditions on $X$ that will allow us to compare $f$ and $h$ on $R_p(X)$. We shall say that $X$ satisfies a $(p,p)$-Sobolev inequality if there exists a constant $C$ such that
\[ \Vert f \Vert_p \leq C \Vert \vert \nabla f \vert \Vert_p \]
holds for all compactly supported functions $f \in N^{1,p}(X)$. We also assume that the volume of all balls of a fixed positive radius $r$ is bounded below by a positive constant. The reason for this assumption is that in next proposition we will need the following inequality, which was derived in \cite[Lemma 6.4]{HoloLangVaha07}, for $f \in L^p(X)$,
\[ \sup_{B_r(x)} \vert f \vert^p \leq C \left( \int_{B_r(x)} \vert f \vert^p du \right)^d, \]
where $C$ is a positive constant and $ 0 < d < 1$.
\begin{Prop} \label{bpboundary} 
Let $f \in BD^p(X)$ and suppose $X$ satisfies a $(p,p)$-Sobolev inequality. Then there exists a function $h \in HBD^p(X)$ that satisfies $h = f$ on $R_p(X)$.
\end{Prop}
\begin{proof}
Let $f \in BD^p(X)$ and let $(\Omega_k)$ be an exhaustion of $X$. By Proposition \ref{harmoniconx} we have a sequence $(h_i)$ such that $h_i$ is $p$-harmonic on $\Omega_i, h_i = f$ on $X\setminus \Omega_i$ and $(h_i)$ converges locally uniformaly to a $p$-harmonic function $h \in HBD^p(X)$. So $f - h \in BD^p(X)$. We will now show that $f - h \in L^p(X).$ For $i > 1, h_i= f$ on $X\setminus \Omega_i$, thus 
\[ \int_{\Omega_i} \vert \nabla h_i \vert^p d\mu \leq \int_{\Omega_i} \vert \nabla h_1 \vert^p d\mu. \]
Consequently,
\[ \vert \nabla(f - h_i) \vert \leq \vert \nabla f \vert + \vert \nabla h_1\vert < M < \infty \]
for all $i$. So $\vert \nabla (f-h_i) \vert \in L^p(X)$ and $f - h_i$ is a function of compact support in $N^{1,p}(X)$. Because $X$ satisfies a $(p,p)$-Sobolev inequality we have for each $i \in \mathbb{N}$,
\[ \Vert f - h_i \Vert_p \leq C \Vert \vert \nabla(f-h_i)\vert \Vert_p \leq C( \Vert \vert \nabla f\vert \Vert_p + \Vert \vert \nabla h_1 \vert \Vert_p ). \]

Thus the sequence $(f - h_i)$ is bounded in $L^p(X)$. Set $u_i = f - h_i$. By passing to a subsequence if necessary the sequence $(u_i)$ converges weakly to $u \in L^p(X)$. By Mazur's lemma, see \cite[Lemma 6.2]{2Bjornbook}, there exists finite convex combinations $v_k = \sum_{i = k}^{N_k} a_{k,i} u_i$ such that $v_k \rightarrow u$ q.e.. So $u = f-h$ q.e.  since $(f - h_i)$ converges pointwise to $f-h$.  Thus $f-h \in L^p(X)$. Denote $f - h$ by $g$. Because $f, h$ and $g$ are all in $BD^p(X)$, they can be extended to continuous functions on $\overline{X}$. Let $x \in R_p(X)$, then $f(x) = (g +h)(x)$. To finish the proof of the proposition we will show $g(x) =0$. Let $(x_k)$ be a sequence in $X$ that satisfies $(x_k) \rightarrow x$ in $\overline{X}$. Since $g \in L^p(X)$ there exists a sequence $(g_n)$ of continuous functions with compact support such that $\Vert g - g_n \Vert_p \rightarrow 0$. Let $\epsilon > 0$, then there exists a natural number $N$ such that if $ n > N, \Vert g - g_n \Vert_p^p < \epsilon/2$. Fix an $n > N$. Due to $g_n$ having compact support, $x_k \notin \mbox{supp}\, g_n$ for all $k > M$, where $M$ is a natural number. Let $r > 0$ satisfy $B_{r}(x_k) \cap \mbox{supp}\, g_n = \emptyset$. Then for $k > M$ we obtain
\[ \int_{B_{r}(x_k)} \vert g (y) \vert^p \, d\mu = \int_{B_{r}(x_k)} \vert (g - g_n)(y) \vert^p \, d\mu \leq \Vert g - g_n\Vert_p^p < \epsilon/2. \]
Consequently, $\vert g(x_k) \vert < \epsilon/2$ for all $k > M$. By continuity of $g$ on $\overline{X}$ we obtain $\vert g(x) \vert < \epsilon$. Thus $g(x)=0$. Therefore, $f(x) = h(x)$ for each $x \in R_p(X)$.
\end{proof} 
\begin{Rem}
We have strong reason to believe $g \in \overline{BD^p_c(X)}_{BD^p}$ (see Section \ref{someremarks}), but we are unable to prove it.
\end{Rem}

We can now prove Theorem \ref{mainresult}. For convenience we restate the result.
\begin{Thm}\label{extendftopharm}
Let $f$ be a continuous real-valued function on $R_p(X)$. Suppose that $X$ satisfies a $(p,p)$-Sobolev inequality. Also assume that the volume of all balls of a fixed radius $r >0$ is bounded below by a positive constant. Then there exists a $p$-harmonic function $h$ on $X$ such that $\lim_{j \rightarrow \infty} h(x_j) = f(x)$ whenever $x \in R_p(X)$ and $(x_j)$ is a sequence in $X$ converging to $X$. 
\end{Thm}
\begin{proof}
Let $f$ be a continuous function on $R_p(X)$. By Tietze's extension theorem there exists a continuous extension of $f$, which we shall also denote by $f$, to all of $\overline{X}$. Recall that we saw in Section \ref{defnpharmboundary} that  $BD^p(X)$ is dense in $C(\overline{X})$ with respect to the uniform norm, which means that $(f_n) \rightarrow f$ in $C(\overline{X})$ only if $\Vert f_n - f \Vert_{\infty} \rightarrow 0$. Let $(f_n)$ be a sequence in $BD^p(X)$ such that $(f_n) \rightarrow f$ in $C(\overline{X}).$ So for $\epsilon > 0$ there exists $N$ for which $\sup_X \vert f_n - f_m \vert < \epsilon$ for $n,m >N$. Let $(\Omega_i)$ be an exhaustion of $X$ by relatively compact sets. For each $n$, Proposition \ref{bpboundary} shows that there exists an $h_n\in HBD^p(X)$ for which $h_n = f_n$ on $R_p(X)$. Furthermore, we also saw in the proof of Proposition \ref{bpboundary} that for each $n$, there exists a sequence $(h_{i,n})$ such that $(h_{i,n}) \rightarrow h_n$ locally uniformly, where $h_{i,n}$ is $p$-harmonic on $\Omega_i$ and $h_{i,n} = f_n$ on $X \setminus \overline{\Omega_i}$. Thus
\[ \sup_{\partial \Omega_i} \vert h_{i,n} - h_{i,m} \vert < \epsilon \]
for all $i \in \mathbb{N}$. Combining $h_{i,m} - \epsilon < h_{i,n} < h_{i,m} + \epsilon $
with the comparison principle, \cite[Theorem 9.39]{2Bjornbook} yields
\[ \sup_{\Omega_i} \vert h_{i,n} - h_{i,m} \vert < \epsilon \]
for each $ i \in \mathbb{N}$. Since $(h_{i,n}) \rightarrow h_n$ and $(h_{i,m}) \rightarrow h_m$ it follows that 
\[ \sup_{\Omega} \vert h_n - h_m \vert < 3\epsilon \]
for any relatively compact subset $\Omega$ of $X$. Consequently 
\[ \sup_X \vert h_n - h_m \vert \leq 3\epsilon. \]
Thus the Cauchy sequence $(h_n)$ converges locally uniformly to a continuous function $h$ on $X$. Also, \cite[Theorem 9.36]{2Bjornbook} tells us that $h$ is $p$-harmonic on $X$.

Let $x \in R_p(X)$ and let $(x_j)$ be a sequence in $X$ for which $(x_j) \rightarrow x$. Choose $\epsilon > 0$. Then there exists an $N$ such that if $n > N$,
\[ \sup_X \vert f_n - f \vert < \epsilon/3 \mbox{ and } \sup_X \vert h_n - h \vert < \epsilon/3. \]
We also have for sufficiently large $j$
\[ \vert h_n(x_j) - f_n(x) \vert < \epsilon/3, \]
because $f_n = h_n$ on $R_p(X)$ and both $f_n$ and $h_n$ are continuous on $X$. It now follows that
\[ \vert h(x_j) - f(x) \vert < \epsilon. \]
Thus $\lim_{j \rightarrow \infty} h(x_j) = f(x)$ and the theorem is proved.
\end{proof}

\section{Some Remarks}\label{someremarks}
In this section we speculate about what would happen if the $(p,p)$-Sobolev condition is dropped from the hypotheses of Proposition \ref{bpboundary}. Recall that in this paper $X$ is assumed to be a locally compact, noncompact, complete metric measure space that satisfies a $(1,p)$-Poincar\'e inequality and whose measure is doubling. In the setting of Proposition \ref{bpboundary} the $p$-harmonic boundary of $X$ coincides with the $p$-Royden boundary of $X$ due to the assumption $X$ satisfies a $(p,p)$-Sobolev inequality, see Proposition \ref{vanish} below. It would be interesting to see what would be true if the $(p,p)$-Sobolev inequality hypothesis is removed from Proposition \ref{bpboundary}. We believe the following to be true:
\begin{Conj} \label{removesobolevcond}
Let $1 < p \in \mathbb{R}$ and let $X$ be a metic measure space. If $f$ is continuous function on $\Delta_p(X)$, then there exists a $p$-harmonic function $h$ on $X$ such that $\lim_{n \rightarrow \infty} h(x_n) = f(x)$, where $x \in \Delta_p(X)$ and $(x_n)$ is a sequence in $X$ that converges to $x$. Moreover, if $f \in BD^p(X)$, then $h \in HBD^p(X)$ and $h = f$ on $\Delta_p(X)$. 
\end{Conj}

A positive answer to this conjecture would extend known results for complete Riemannian manifolds of bounded geometry \cite[Theorem 1]{Lee05} and connected graphs of bounded degree \cite[Theorem 2.6]{Puls10}.

Conjecture \ref{removesobolevcond} can be proven if a $p$-Royden decomposition can be obtained for functions in $BD^p(X)$. More specifically, given $f \in BD^p(X)$ there exists an unique $g \in \overline{BD^p_c(X)}_{BD^p}$ and a unique $h \in HBD^p(X)$ such that
\begin{equation}
f = g + h. \label{eq:proydendecomp}
\end{equation}

Using the $p$-Royden decomposition we immediately have $f(x) = h(x)$ for each $x \in \Delta_p(X)$ since $g(x) = 0$ for all $x \in \Delta_p(X)$. In fact, in our proof of Proposition \ref{bpboundary} we showed that if $f \in BD^p(X)$, then there exists  $g \in L^p(X)$ and $h \in HBD^p(X)$ such that $f = g+h$. Furthermore, we saw that there was a sequence $(f - h_n)$ in $BD^p_c (X)$ that converged pointwise a.e. to $g$. Actually the following is true
\begin{Prop} \label{simnonamencase}
Suppose $X$ satisfies the hypotheses of Proposition \ref{bpboundary}. Then $\overline{BD^p_c(X)}_{BD^p} \subseteq L^p(X)$. Furthermore, any sequence $(f_n)$ in $BD^p_c(X)$ that converges to $f$ in the $BD^p$-topology also satisfies $\Vert f_n - f \Vert_p \rightarrow 0$ as $n \rightarrow \infty$.
\end{Prop}
\begin{proof}
Let $f \in \overline{BD^p_c(X)}_{BD^p}$ and let $(f_n)$ be a sequence in $BD^p_c(X)$ that converges to $f$ in the $BD^p$-topology. Then $(f_n) \rightarrow f$ uniformly on compact sets and $\int_X \vert \nabla(f - f_n) \vert^p d\mu \rightarrow 0$ as $n \rightarrow \infty$. By the $(p,p)$-Sobolev condition there exists a constant $C$ such that 
\[ \Vert f_j - f_k \Vert_p \leq C \Vert \vert \nabla(f_j - f_k)\vert \Vert_p, \]
where $f_j$ and $f_k$ belong to the sequence $(f_n)$. It now follows from
\[ \vert \nabla(f_j - f_k) \vert \leq \vert \nabla(f_j - f)\vert + \vert \nabla(f-f_k) \vert \]
and $\int_X \vert \nabla(f - f_n) \vert^p d\mu \rightarrow 0$ that $(f_n)$ is a Cauchy sequence in $L^p(X)$. Let $f'$ be the limit of $(f_n)$ in $L^p(X)$. Then $(f_n)$ also converges to $f'$ pointwise a.e.. Consequently, $f = f'$ a.e. and $f \in L^p(X)$ as desired.
\end{proof}
We are now able to prove
\begin{Prop} \label{vanish}
Suppose $X$ satisfies the hypothesis of Proposition \ref{bpboundary}. If $f \in \overline{BD^p_c(X)}_{BD^p}$, then $f(x) = 0$ for all $x \in R_p(X)$.
\end{Prop}
\begin{proof}
Let $x \in R_p(X)$ and let $f \in \overline{BD^p_c(X)}_{BD^p}$. Suppose $(x_k)$ is a sequence in $X$ for which $(x_k) \rightarrow x$ in $\overline{X}$; and let $(f_n)$ be a sequence in $BD^p_c(X)$ such that $(f_n) \rightarrow f$ in the $BD^p$-topology. Let $\epsilon > 0$, now by Proposition \ref{simnonamencase} there exists a natural number $N$ that satisfies $\Vert f - f_n \Vert_p^p < \epsilon/2$ for all $n>N$. Using the argument from the last paragraph of the proof Proposition \ref{bpboundary} we obtain that $f(x) = 0$. 
\end{proof}
A direct consequence of the above proposition is
\begin{Cor} \label{agree}
If $X$ satisfies the $(p,p)$-Sobolev inequality, then $R_p(X) = \Delta_p(X).$
\end{Cor}
Now if $x \in R_p(X) \setminus \Delta_p(X)$ then there would exists a $g \in \overline{BD^p_c(X)}_{BD^p}$ for which $g(x) \neq 0$. Let $h \in HBD^p(X)$ be such that $h(x) \neq 0$. Then $f = g+h \in BD^p(X)$ and assuming that the $p$-Royden decomposition (\ref{eq:proydendecomp}) is unique, we see that $f(x) \neq h(x)$ and the conclusion of Proposition \ref{bpboundary} would be false. Thus if the $(p,p)$-Sobolev inequality condition is dropped from the hypotheses of Proposition \ref{bpboundary}, in all likelihood $R_p(X) \neq \Delta_p(X)$ and $R_p(X)$ would need to be replaced with $\Delta_p(X)$ in the statement of the proposition for it to be true. 

Assuming the uniqueness of the $p$-Royden decomposition (\ref{eq:proydendecomp}) for each $f \in BD^p(X)$, the following Liouville type result can be obtained. The set $HBD^p(X)$ contains nonconstant functions precisely when $\vert \Delta_p(X) \vert > 1$. To see this, suppose $x_1$ and $x_2$ are distinct elements in $\Delta_p(X)$. Choose $f \in BD^p(X)$ such that $f(x_1) \neq f(x_2)$. Then the $h$ in the $p$-Royden decomposition of $f$ yields a nonconstant function in $HBD^p(X)$. If $\vert \Delta_p(X) \vert = 1$, then the uniqueness of the $p$-Royden decomposition for each $f$ in $BD^p(X)$ ensures that $HBD^p(X)$ consists entirely of constant functions. 
\bibliographystyle{plain}
\bibliography{proyden_dirichletmms}
\end{document}